\documentclass[11pt]{article}
\usepackage{amsmath,amsthm,verbatim,amssymb,amsfonts,amscd,graphicx}
\usepackage{enumerate}
\usepackage{mathrsfs}
\usepackage{mathtools}
\usepackage{cite}
\usepackage{appendix}
\usepackage{hyperref}
\newtheorem{thm}{Theorem}[section]
\newtheorem{cor}[thm]{Corollary}
\newtheorem{prop}[thm]{Proposition}
\newtheorem{lem}[thm]{Lemma}

\theoremstyle{definition}
\newtheorem{defn}[thm]{Definition}

\newtheorem{exmp}[thm]{Example}

\theoremstyle{remark}
\newtheorem{rem}[thm]{Remark}

\makeatletter
\let\c@equation\c@thm
\makeatother
\numberwithin{equation}{section}

\title{On the Speed of an Excited Asymmetric Random Walk}
\author{Mike Cinkoske, Joe Jackson, Claire Plunkett}

\begin{document}
\maketitle

\begin{abstract}
An excited random walk is a non-Markovian extension of the simple random walk, in which the walk's behavior at time $n$ is impacted by the path it has taken up to time $n$. The properties of an excited random walk are more difficult to investigate than those of a simple random walk. For example, the limiting speed of an excited random walk is either zero or unknown depending on its initial conditions. While its limiting speed is unknown in most cases, the qualitative behavior of an excited random walk is largely determined by a parameter $\delta$ which can be computed explicitly. Despite this, it is known that the limiting speed cannot be written as a function of $\delta$. We offer a new proof of this fact, and use techniques from this proof to further investigate the relationship between $\delta$ and limiting speed. We also generalize the standard excited random walk by introducing a ``bias" to the right, and call this generalization an excited asymmetric random walk. Under certain initial conditions we are able to compute an explicit formula for the limiting speed of an excited asymmetric random walk.
\end{abstract}

\section{Introduction}

A simple random walk is a discrete Markovian model of random motion whose properties are well understood. More specifically, a simple random walk $(W_{n})_{n\geq 0}$ is a Markov chain with transition probabilities
\begin{align*} p(j,k) = \text{P} \left( W_n = k \ | \ W_{n-1} = j \right) = \begin{cases} 
p & k = j+1 \\
1-p & k = j-1 \\
0& \text{otherwise.}\\
\end{cases}
\end{align*}
We can give the following informal interpretation of a simple random walk: a random walker starts at 0 and takes an infinite sequence of independent steps. Each step is to the right with probability $p \in (0,1)$ and to the left with probability $1-p$. Alternatively, a simple random walk $(W_{n})_{n\geq 0}$ can be defined as $ W_n = \sum_{i=1}^n \omega_i $ where the $ \omega_i $ are i.i.d. random variables with 
\begin{align*}
& \text{P} \left( \omega_i = 1 \right) = p
\\
& \text{P} \left( \omega_i = -1 \right) = 1- p.
\end{align*}

Because simple random walks can be represented both as Markov chains and as sums of i.i.d. random variables, their properties are well understood. For example, one-dimensional simple random walks are recurrent, i.e. they return to $0$ infinitely many times with probability $1$, if and only if $p = \frac{1}{2}$. Recall that if a Markov chain is not recurrent, it is transient, i.e. returns to $0$ only a finite number of times. We define the limiting speed of any random walk $(R_{n})_{n \geq 0}$ to be
\begin{align}
	\lim_{n\rightarrow \infty} \frac{R_{n}}{n},
\end{align}
and note that this definition of limiting speed applies also to the more complex random walk variations described below. For simple random walks, the following proposition is known:

\begin{prop}
The limiting speed of a simple random walk $(W_{n})_{n\geq 0}$ with probability $p$ of stepping to the right is $2p - 1$.
\end{prop}

\begin{proof}
  We have from above that
  \begin{align*}
    \lim_{n\rightarrow \infty} \frac{W_{n}}{n} = \lim_{n \rightarrow \infty} \frac{\sum_{i=1}^n \omega_i }{n}.
  \end{align*}
  Applying the Strong Law of Large Numbers, we obtain
  \begin{align*}
    \lim_{n \rightarrow \infty} \frac{\sum_{i=1}^n \omega_i }{n} = \mathbb{E}[\omega_1] = 2p - 1.
  \end{align*}
\end{proof}

Simple random walks have diverse applications to real world problems, but some applications are better modeled by processes which are non-Markovian. To that effect, mathematicians have introduced a number of self-interacting random walks. One variation is the excited random walk, first introduced by Benjamini and Wilson \cite{benjamini} in 2003 and later generalized by Zerner \cite{zerner} and Kosygina and Zerner \cite{kosygina2}.

An excited random walk is a non-Markovian extension of a simple random walk which can be described informally as follows: at each site on the number line, we place $ M $ cookies, each of which has a ``strength." The random walker starts at the origin and takes an infinite sequence of steps. The probability distribution of each step depends on the number of cookies left at the walker's current location, and when the walker leaves a site with cookies remaining, he eats a cookie (See Figure \ref{fig:erw}).
\begin{figure}[h!]
	\centering
	\includegraphics[width=1\linewidth]{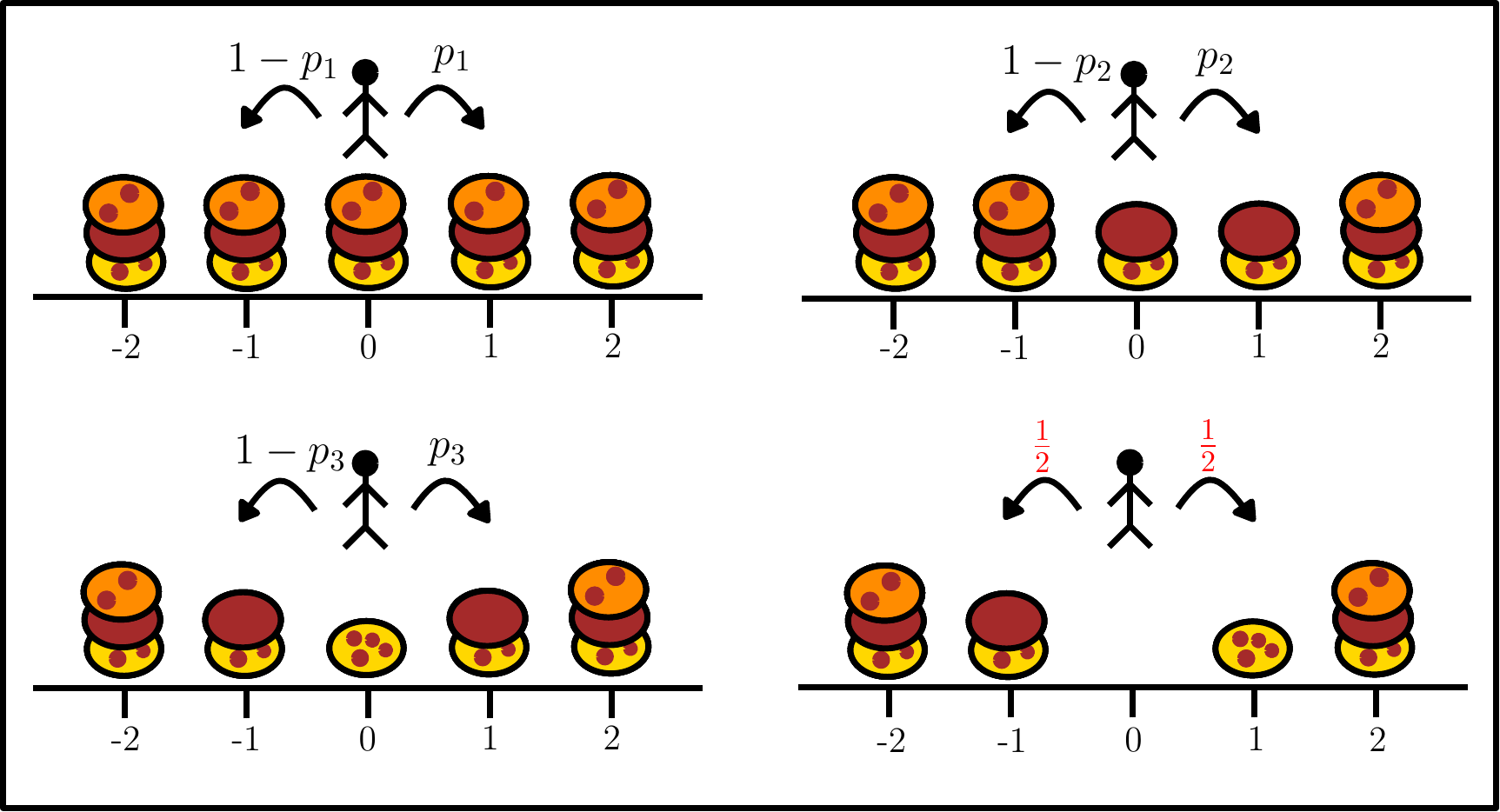}
	\caption{Excited Random Walk with 3 Cookies}
	\label{fig:erw}
\end{figure}

Mathematically, we specify the number of cookies $M$ and a vector of cookie strengths $\mathbf{p} \in \mathbb{R}^{M}$ with $p_{i} \in (0,1), \ i \in \{1,2,...,M\}$. We let $(Y_{n})_{n\geq 0}$ be an excited random walk and define the probabilities of stepping right or left given the first $n$ steps of the walk to be 
\begin{align*}
& \text{P} \left( Y_{n+1} - Y_n = 1 \ | \ Y_0, Y_1, \ldots, Y_n \right) = 
\begin{cases}
p_i \quad \text{if } \# \left\{j; Y_{j} = Y_n \right\} = i \leq M
\\
1/2 \quad \text{otherwise,}
\end{cases}
\\
& \text{P} \left( Y_{n+1} - Y_n = -1 \ | \ Y_0, Y_1, \ldots, Y_n \right) = 1 - \text{P} \left( Y_{n+1} - Y_n = 1 \ | \ Y_0, \ldots, Y_n \right).
\end{align*}
That is, the walker's probability of stepping right on his $i^{th}$ visit to a site is given by the strength of the $i^{th}$ cookie if $i \leq M$, and is $\frac{1}{2}$ if $i>M$. In this model, the probability of stepping left or right at time $n$ depends on the path the walker took up to the current time, and hence $(Y_{n})_{n \geq 0}$ is neither a Markov chain nor a sum of i.i.d. random variables. This makes analyzing its asymptotic behavior, such as its recurrence or limiting speed, difficult. Nevertheless, the following theorem has been proven, which shows how the qualitative behavior of an excited random walk is determined by the parameter $\delta(M,\mathbf{p}) $, defined as 
\begin{align}
\delta(M,\mathbf{p}) = \sum_{i = 1}^{M} (2p_{i} - 1).
\label{delta}
\end{align}

\begin{thm}[Zerner \cite{zerner}, Basdevant and Singh \cite{bassingh}, Kosygina and Zerner \cite{kosygina2}]
	A standard excited random walk with $M$ cookies and cookie strength vector $\mathbf{p}$ is transient to the right if and only if $\delta(M,\mathbf{p}) > 1$. It has positive speed if and only if $\delta(M,\mathbf{p}) > 2$. For $ -2 \leq \delta(M, \mathbf{p} ) \leq 2 $, the walk has zero speed.
	
	\label{verybasics}
\end{thm}
 \noindent While this theorem suggests $\delta(M,\mathbf{p})$ might determine the speed of an excited random walk, it is known that the speed of an excited random walk cannot be written as a function of $\delta(M,\mathbf{p})$ \cite{holmes}. In Section 4 we present a new proof of this fact. We use the techniques in our proof to give several concrete examples of known monotonicity properties, and we further show that $\delta(M,\mathbf{p})$ and $v(M,\mathbf{p})$ are unrelated when $\delta > 2$, in the sense that there exists an excited random walk with arbitrarily large $\delta$ parameter and arbitrarily small speed.

\subsection{What is an Excited Asymmetric Random Walk?}

We now introduce the primary object of analysis in this paper: the excited asymmetric random walk. An excited asymmetric random walk is a generalization of the excited random walk, in which the probabilities of stepping left or right from a site with no cookies need not be $\frac{1}{2}$ (see Figure \ref{fig:earw}).
\begin{figure}[h!]
	\centering
	\includegraphics[width=1\linewidth]{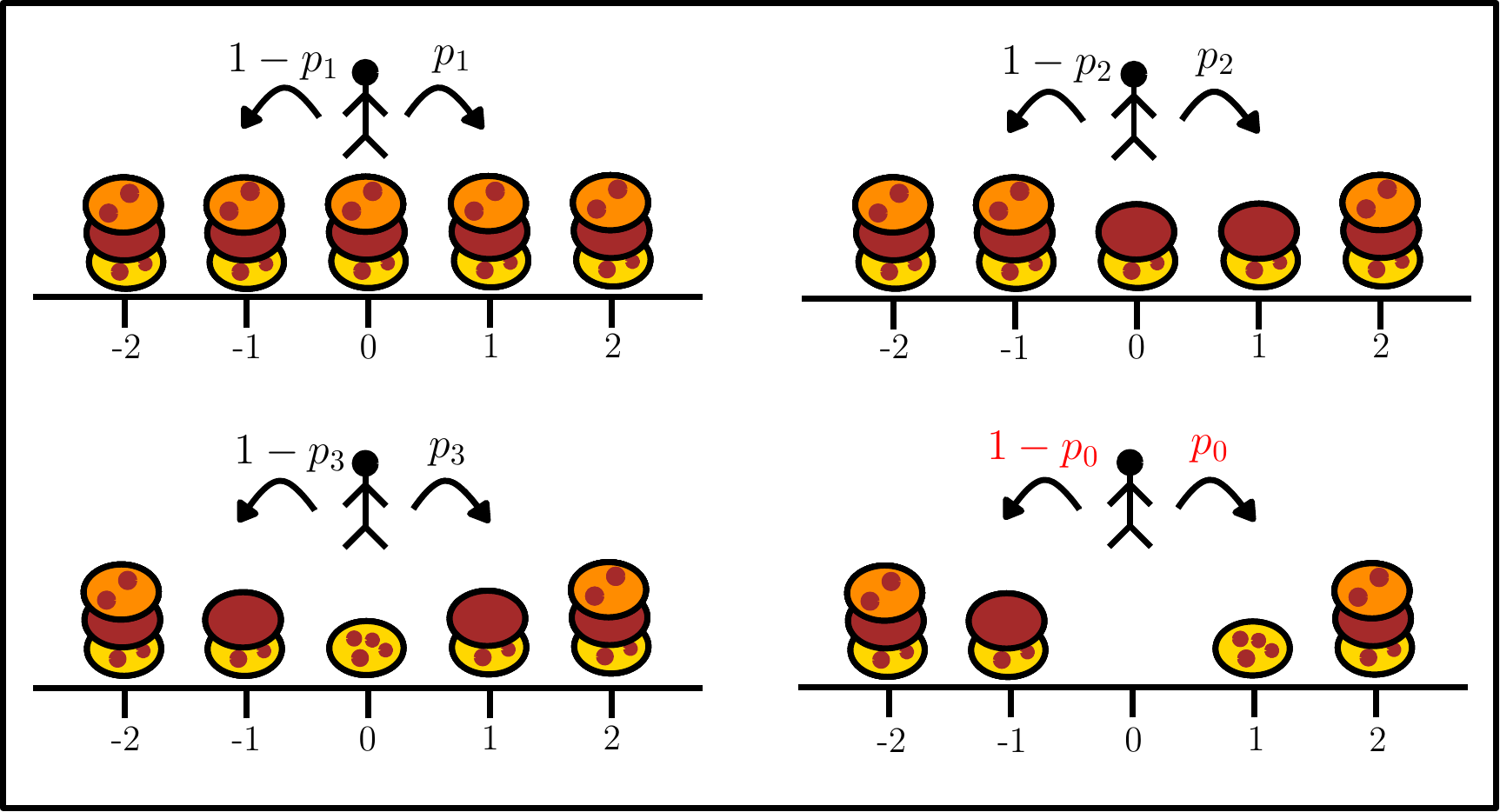}
	\caption{Excited Asymmetric Random Walk with 3 Cookies}
	\label{fig:earw}
\end{figure}
More formally, if $(X_{n})_{n \geq 0}$ is an excited asymmetric random walk, we specify a number of cookies $M$, a vector of cookie strengths $\mathbf{p} \in \mathbb{R}^{M}$, and a bias parameter $p_{0}$. Then the probability of stepping right or left given the first $n$ steps of the walk is 
\begin{align*}
& \text{P} \left( X_{n+1} - X_n = 1 \ | \ X_0, X_1, \ldots, X_n \right) = 
\begin{cases}
p_i \quad \text{if } \#\{j; X_{j} = X_n\} = i \leq M
\\
p_0 \quad \text{otherwise,}
\end{cases}
\\
& \text{P} \left( X_{n+1} - X_n = -1 \ | \ X_0, X_1, \ldots, X_n \right) = 1 - \text{P} \left( X_{n+1} - X_n = 1 \ | \ X_0, \ldots, X_n \right).
\end{align*}
We will assume throughout this paper that $p_{0} > \frac{1}{2}$; a symmetry argument extends our analysis to the other case.

Our motivation for studying this type of random walk is as follows: in a standard excited random walk, the speed function is known to be zero when the number of cookies $M$ is less than 3, since $\delta \leq 2$ if $ M < 3$ by \eqref{delta}, and has so far been too difficult to compute exactly when $M \geq 3$ unless the speed is zero. Adding a drift to the excited random walk makes the speed function nontrivial, even when $M$ is small (See Lemma \ref{thelemma} below). In the case of $M = 1$, we can compute the speed explicitly.

\begin{thm}
	The limiting speed of an excited asymmetric random walk with one $p_1$ cookie and bias parameter $p_0$, with $p_{1} \in (0,1), p_{0} \in ( 1/2,1)$ is given by
	\begin{align}
	v^{*} \left(p_0,p_1 \right) = \frac{2p_0 - 1}{2p_0 - 1 + 2(1-p_1)}. \label{speedeq}
	\end{align}
	
	\label{speedthm}
	
\end{thm}

As an illustration of this speed, we have Figure \ref{fig:speed} which shows the speed as a function of $p_0 \in (.5,1)$ for $p_1 = 0.8, \ 0.9, \ 0.99 $.

\begin{figure}[h!]
	\centering
	\includegraphics[width=.7\linewidth]{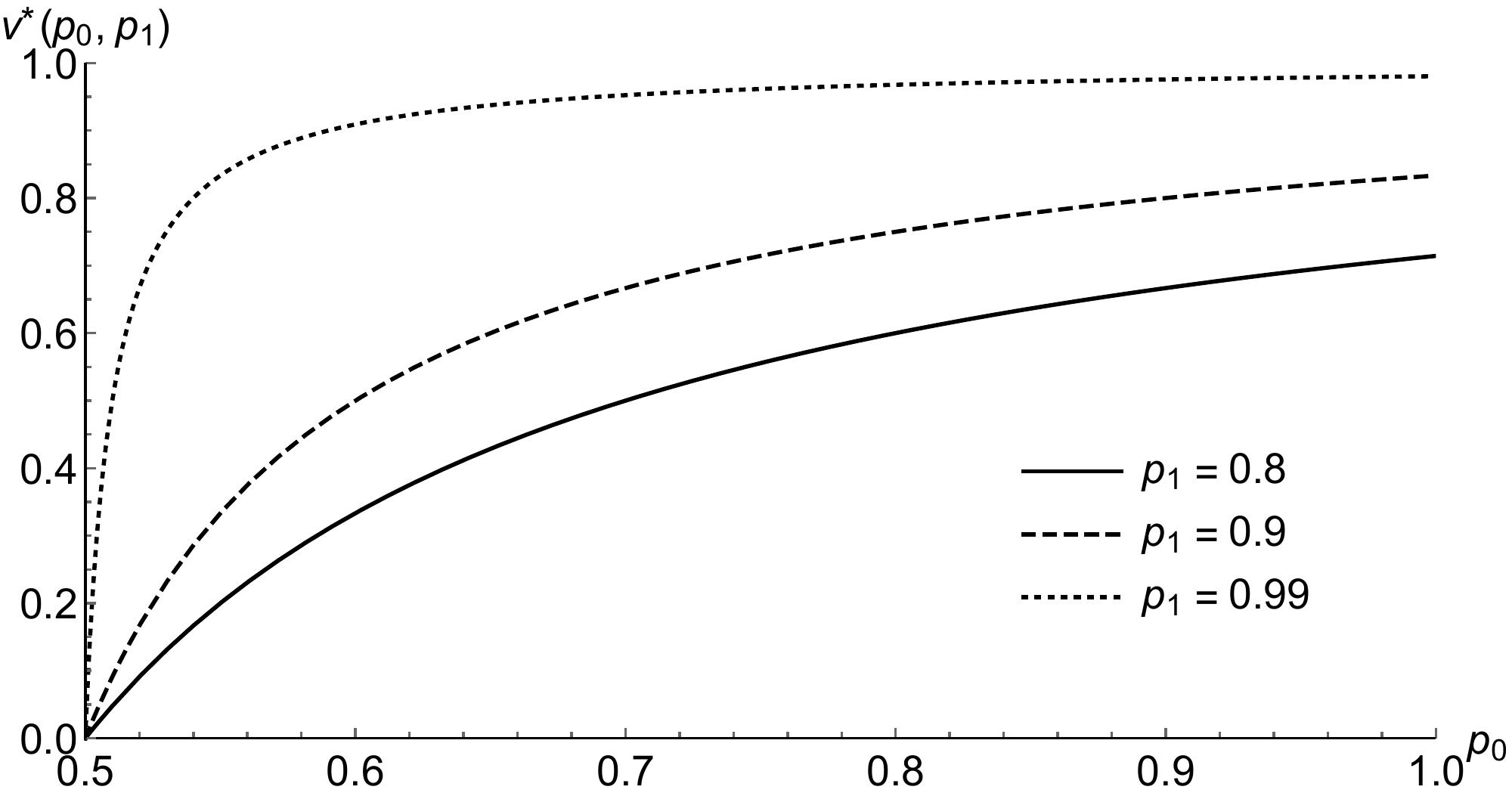}
	\caption{$ v^{*} \left(p_0, p_1 \right) $ for three values of $p_1$}
	\label{fig:speed}
\end{figure}

Before giving the proof of our main results, first we will review some results about standard excited random walks that are needed for our analysis.

\subsection{An Associated Markov Chain}

While Theorem \ref{verybasics} gives important information about the qualitative behavior of excited random walks, quantitative analysis will require a more precise probabilistic statement of the speed. To this end, we introduce what is known in the literature as the backwards branching-like process associated to a random walk $ ( X_n )_{n \geq 0} $. The backwards branching-like process $( Z_n )_{n \geq 0}$ is Markovian and can be associated to both standard excited random walks and excited asymmetric random walks, so in the remainder of this section we will use the term excited random walk to refer to both variations.  We begin by defining the random variables 
\begin{align}
& T_n = \inf_{t \geq 0} \{ t : X_t = n \}, \label{tndef}
\\
& U_x^n = \# \{ t < T_n : X_t = x, X_{t+1} = x - 1 \}.
\end{align}
$T_{n}$ is interpreted as the hitting time of site $n$, while $U_x^n$ is the number of left steps from site $x$ by time $T_{n}$ (See Figure \ref{fig:bbp}). Under this definition, it is clear that $U_{n}^{n} = 0$ and that $U_{x}^{n}$ are random variables which are non-decreasing in $n$. We can think of $U_x^n$ as the number of left steps from $x$ before reaching $x + 1$ for the first time plus the number of left steps from $x$ between the first left step from $x+1$ and the first return to $x+1$ before $T_n$, plus the number of left steps from $x$ between the second left step from $x+1$ and the second return to $x+1$ before $T_n$, and so on. Thus the distribution of the random variable $ U_x^n $ is determined by the value of $ U_{x+1}^n $, and the process $(U_{n}^{n}, U_{n-1}^{n}, ..., U_{0}^{n})$ is a Markov chain. The transition probabilities are given by 
\begin{align}
p(l,m) & = \text{P} \left( U_x^n = m | U_{x+1}^n = l \right) \nonumber
\\
&= \text{P} \left( \text{$m$ steps left from $x$ by $T_n \ | \ l$ steps left from $x+1$ by $T_n$} \right) \nonumber
\\
&= \text{P} \left( \text{$m$ steps left from $x$ before $l+1$ steps right from $x$} \right).
\label{transitions}
\end{align}
The last equality follows from the the fact that if the walk takes $l$ left steps from $x + 1$, the walk must take $l$ steps right from $x$ to return to $x + 1$ after each left step, plus a step right from $x$ to reach $x + 1$ for the first time.

\begin{figure}[h!]
	\centering
	\includegraphics[width=.9\linewidth]{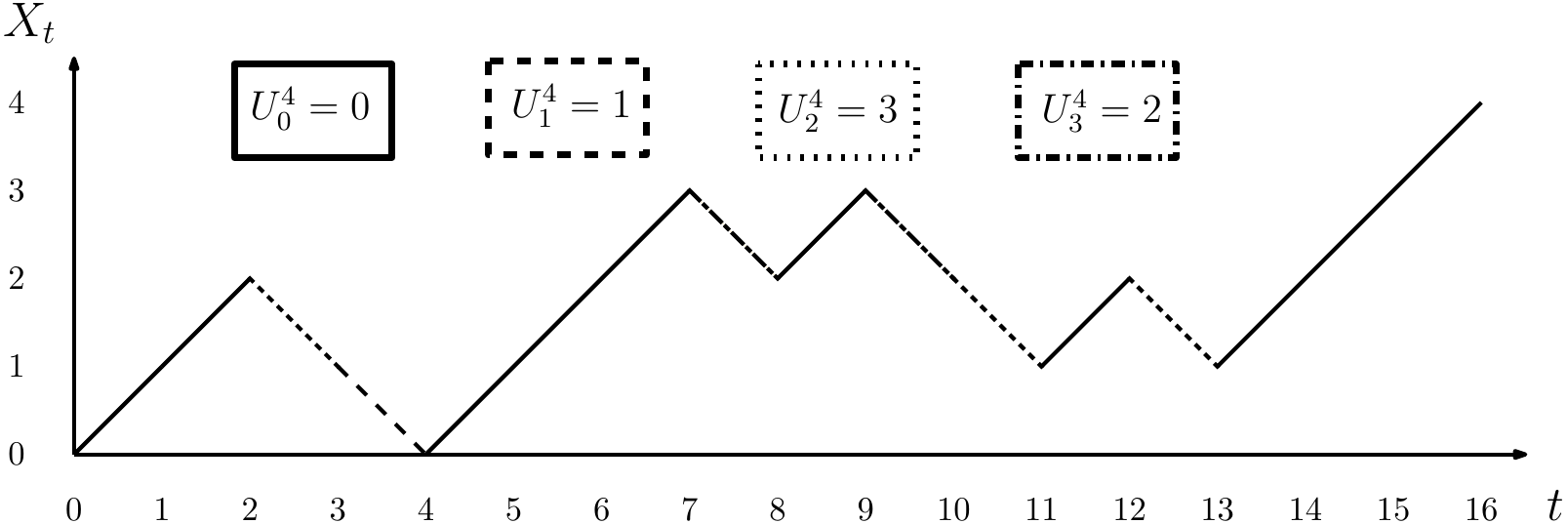}
	\caption{An example of $U_{x}^{4}$.}
	\label{fig:bbp}
\end{figure}

The backwards branching-like process $ \left(Z_n \right)_{n \geq 0} $, which will be constructed more rigorously in the remainder of this section, is a Markov chain designed  such that the following theorem holds.

\begin{thm}[Basdevant and Singh Proposition 2.2 \cite{bassingh}]
	If an excited random walk is recurrent or transient to the right, then for all $ n \in \mathbb{N} $, if $Z_0 = 0$, then the processes $ \left( Z_0, Z_1, \ldots, Z_n \right) $ and $ \left( U_n^n, U_{n-1}^n, \ldots, U_0^n \right) $ have the same distribution.
	\label{ZandU}
\end{thm}

\begin{rem}
The backwards branching-like process $ \left(Z_n \right)_{n \geq 0} $ is called branching-like because the sequence $ \left( U_n^n, U_{n-1}^n, \ldots, U_0^n \right) $ is like a branching process with migration, in which $ U_{x-1}^n $ is the result of $ U_x^n +1 $ individuals reproducing according to the probability distributions explained below. The process $( Z_n )_{n \geq 0}$ is called backwards because the index of the Markov chain $(U_{n}^{n}, U_{n-1}, ..., U_{0}^{n})$ starts at $n$ and decreases.
\end{rem}

To construct $ \left( Z_n \right)_{n \geq 0} $, we first define independent sequences of random variables $ \left( \xi_{n,j} \right)_{j \geq 0} $ by the probability of jumping right on the $j^{th}$ visit to site $n$ in the corresponding random walk. For example, in a standard excited random walk with $M$ cookies and cookie vector $\mathbf{p}$, we have
\begin{align*}
\xi_{n,j} = & \begin{cases}
\text{Bernoulli}(p_j) : j \leq M \\
\text{Bernoulli}(1/2) : j > M.
\end{cases}
\end{align*}
Then, we define the random variables $ A_{i,k} $ to be the number of ``failures" in the sequence of $ \left( \xi_{i,j} \right)_{j \geq 0} $ before $ k $ ``successes." That is,
\begin{align}
A_{i,k} = \min \left\{ n \geq 0 : \sum_{j=1}^{n+k} \xi_{i,j} = k \right\}.
\label{Aik}
\end{align}
Then the backwards branching-like process $ \left( Z_n \right)_{n \geq 0} $ is defined to be a Markov chain with transition probabilities 
\begin{align*}
& p(l,m) = \text{P} \left( Z_{n+1} = m | Z_n = l \right) = \text{P} \left( A_{n+1,l+1} = m \right).
\end{align*} 

We omit a proof of Theorem \ref{ZandU}, but the connection between the processes $ \left( Z_0, Z_1, \ldots, Z_n \right) $ and $ \left( U_n^n, U_{n-1}^n, \ldots, U_0^n \right) $ can be seen by considering definitions \eqref{transitions} and \eqref{Aik} and interpreting left steps as failures and right steps as successes.

\begin{exmp}
	\label{trans-probs}
	For an excited asymmetric random walk with one $p_1$ cookie and bias parameter $p_0$, we have the following transition probabilities for the associated backwards branching-like process: 
	\begin{align*}
	& p(0,0) = p_1
	\\
	& p(0,k) = (1 - p_1)(1 - p_0)^{k-1}p_0 \quad \text{for } k > 0
	\\
	& p(1,0) = p_1p_0
	\\
	& p(1,k) = p_1(1 - p_0)^kp_0 + k(1 - p_1)(1 - p_0)^{k-1}p_0^2 \quad \text{for } k > 0
	\\
	& p(j,k) = \tbinom{k+j-1}{j-1} p_1(1 - p_0)^kp_0^j + \tbinom{k + j - 1}{j} (1 - p_1)(1 - p_0)^{k-1}p_0^{j+1}
	\\
	& \qquad \text{for } j \geq 2, k \geq 0.
	\end{align*}
	
\end{exmp}

Since the transition probabilities are all non-zero, the backwards branching-like process $ \left( Z_n \right)_{n \geq 0} $ is an irreducible Markov chain. If its unique stationary distribution exists, we will denote it as $\pi$ and we will use P$_\pi$ and $\mathbb{E}_\pi$ to denote probabilities and expectations conditional on $ Z_0 \sim \pi $.

The following lemma gives a useful decomposition of the random variables $A_{i,k}$, which we will use in our proof of Theorem \ref{speedthm}.

\begin{lem}[Basdevant and Singh 2008 \cite{bassingh}]
	\label{decomp}
	
	For all $j \geq M$, we have
	\begin{align*}
	A_{i,k}  \stackrel{\text{dist}}{=} A_{i,M} + \gamma_1 + \cdots + \gamma_{k-M+1}
	\end{align*}
	where $(\gamma_i)_{i \geq 0}$ are i.i.d. geometric random variables independent of $A_{i,M-1}$ with parameter $p_0$, i.e. $P(\gamma_1 = j) = (1 - p_0)^j p_0$, where $p_{0}$ is the probability of stepping right on the $(M+1)^{st}$ visit to a site in the corresponding random walk.
\end{lem}
The proof of Lemma \ref{decomp} is identical to that of Basdevant and Singh's Lemma 2.1 \cite{bassingh} for the standard excited random walk, and is thus omitted.

Now, having developed the backwards branching-like process, we have the machinery necessary to state the following theorem, which gives the probabilistic formulation of the speed which we use to prove Theorem \ref{speedthm}.

\begin{thm}[Basdevant and Singh 2008 \cite{bassingh}]
	For a standard excited random walk with $M$ cookies and cookie strength vector $\mathbf{p}$, the stationary distribution $\pi$ of the associated backwards branching-like process $ \left( Z_n \right)_{n \geq 0} $ exists if and only if $\delta(M,\mathbf{p}) > 1$.
	
	Further, if $\delta(M,\mathbf{p}) > 1$, then the speed of the walk is given by
	\begin{align}
	v(M, \mathbf{p}) = \frac{1}{1 + 2 \mathbb{E}_{\pi} \left[ Z_0 \right] }. \label{speedBS}
	\end{align}
	with the convection that $\frac{1}{+\infty} = 0$.
	
	\label{mediumbasics}
\end{thm}

\begin{rem}
Using Lemma \ref{decomp} and Theorem \ref{ZandU}, it is possible to derive \eqref{speedBS}. While we omit the proof, the argument is guided by the intuition that the speed of the walk at time $T_{n}$ can be related to the number of left steps it has taken by that time, which can be investigated through the asymptotic behavior of $(Z_{n})_{n \geq 0}$. The representation of the speed as given by \eqref{speedBS} can be used to show that the speed is nonzero if and only if $\delta(M, \mathbf{p} ) > 2 $. This is done by showing that an excited random walk with parameters $M,  \mathbf{p}$ has $ \mathbb{E}_{\pi} \left[ Z_0 \right] < \infty $ (and thus positive speed) if and only if $\delta(M,\mathbf{p}) > 2$ \cite{bassingh}.
\end{rem}

Importantly, the proof of Theorem \ref{mediumbasics} requires only that the walk is transient to the right, not that there are finitely many cookies. Since an excited asymmetric random walk can be interpreted as an excited random walk with infinitely many cookies at each location, \eqref{speedBS} is valid for the excited asymmetric random walk model when it is transient to the right.

\section{Calculating the Speed}

In this section we prove Theorem \ref{speedthm}. We will need the following lemma, which we will prove in the next section.

\begin{lem}
	The backwards branching-like process $(Z_{n})_{n \geq 0} $ associated to an excited asymmetric random walk with parameters $M \geq 1, \mathbf{p} \in (0,1)^M$ and bias parameter $p_0 \in (1/2,1)$ has a stationary distribution $\pi$, and $\mathbb{E}_\pi \left[Z_0 \right] < \infty$. \label{thelemma} 
\end{lem}

Let $(Z_{n})_{n \geq 0}$ be the backwards branching-like process associated to an excited asymmetric random walk with one $p_{1}$ cookie and bias parameter $p_{0} > \frac{1}{2}$, and let $\pi$ be its stationary distribution. From \eqref{speedBS}, we know that if we can calculate $\mathbb{E}_\pi[Z_{0}]$, then we can calculate the speed of the walk. Since calculating explicit values of $\pi$ seems to be a very difficult problem (See Appendix A), we instead follow the approach of Basdevant and Singh \cite{bassingh} and attempt to calculate $\mathbb{E}_\pi[Z_0]$ by studying the probability generating function of $\pi$. Let
\begin{align}
G(s) = \mathbb{E}_\pi[s^{Z_0}] = \sum_{k=0}^{\infty}\pi(k)s^k
\end{align}
be the probability generating function of $\pi$. We study the p.g.f. of $\pi$ because of the well-known property that $G^\prime(1) = \mathbb{E}_\pi[Z_0]$, where $G^\prime(1)$ is the left derivative at 1. This enables us to calculate $\mathbb{E}_\pi[Z_0]$ without calculating $\pi$ explicitly.


\subsection{Deriving a Recursive Formula for the P.G.F}

Since explicitly calculating $\pi$, and hence $G(s)$, is a difficult problem, we instead find a recursive formula for $G(s)$.

\begin{prop}
	\label{recursive-prop}
	The probability generating function of $\pi$ satisfies the recursive formula
	\begin{align}
	G(s)  =  \left(\frac{p_1 + s(p_0 - p_1)}{1 - s(1 - p_0)}\right)G\left(\frac{p_0}{1-s(1-p_0)}\right).\label{recursive}
	\end{align}
	
\end{prop}

\begin{proof}
	Since $\pi$ is a stationary distribution, we know that $\mathbb{E}_\pi[s^{Z_0}] = \mathbb{E}_\pi[s^{Z_1}]$, and thus we can also write $G(s)$ as
	
	\begin{align}
	G(s) = \sum_{k=0}^{\infty}\pi(k)\mathbb{E}[s^{Z_1} | Z_0 = k] \label{to-plug-in}.
	\end{align}
	From Lemma \ref{decomp}, we have
	\begin{align*}
	\mathbb{E}[s^{Z_1} | Z_0 = k] = \mathbb{E}[s^{A_{1,1} + \gamma_1 + \cdots + \gamma_k}] = \mathbb{E}[s^{A_{1,1}}]\mathbb{E}[s^{\gamma_1}]^k,
	\end{align*} 
	where $\mathbb{E}[s^{A_{1,1}}] = \mathbb{E}[s^{Z_1} | Z_0 = 0]$ from our definitions. Using the transition probabilities given in Example \ref{trans-probs}, we can calculate $\mathbb{E}[s^{Z_1}| Z_0 = 0]$ as
	\begin{align*}
	\mathbb{E}[s^{Z_1}| Z_0 = 0] = \sum_{k=0}^{\infty}s^kp(0,k) = \frac{p_1 + s(p_0 - p_1)}{1 - s(1 - p_0)}.
	\end{align*}
	Using the p.g.f. of a geometric random variable and substituting into \eqref{to-plug-in}, we have
	\begin{align*}
	G(s) & =  \mathbb{E}[s^{Z_1}|Z_0 = 0]\sum_{k=0}^{\infty}\pi(k)\left(\frac{p_0}{1 - s(1 - p_0)}\right)^k
	\\
	& = \left(\frac{p_1 + s(p_0 - p_1)}{1 - s(1 - p_0)}\right)G\left(\frac{p_0}{1 - s(1 - p_0)}\right).
	\end{align*}
\end{proof}

\subsection{Finding $\mathbb{E}_\pi \left[Z_0 \right]$}

Using Lemma \ref{thelemma} and Proposition \ref{recursive}, we now prove Theorem \ref{speedthm}.

\begin{proof}
Recall that all probability generating functions are differentiable on $ [0,1] $, so $G^\prime(1) = \mathbb{E}_\pi \left[Z_0 \right]$, where $G^\prime(1)$ is the left-hand derivative of $G(s)$ at $1$.
Applying the product and chain rules to \eqref{recursive} yields
\begin{align*}
  G^\prime(s)
  & = \left(\frac{p_1 + s(p_0 - p_1)}{1 - s(1 - p_0)}\right) G^\prime\left(\frac{p_0}{1-s(1-p_0)}\right)\left(\frac{p_0(1-p_0)}{(1-s(1-p_0))^2}\right) \\
  &\quad\quad+ G\left(\frac{p_0}{1-s(1-p_0)}\right)\left(\frac{p_0(1-p_1)}{(1-s(1-p_0))^2}\right).
\end{align*}
Evaluating at $ s = 1$ yields
\begin{align*}
  G^\prime(1)
  & = G^\prime\left(1\right)\left(\frac{1-p_0}{p_0}\right) + \left(\frac{1-p_1}{p_0}\right),
\end{align*}
and solving for $G^\prime(1)$, which is possible since $G^{\prime}(1) = \mathbb{E}_{\pi}[Z_{0}] < \infty$ by Lemma \ref{thelemma}, we obtain
\begin{align*}
G^\prime(1) = \mathbb{E}_\pi \left[Z_0 \right] = \frac{1-p_1}{2p_0 - 1}.
\end{align*}
Substituting into \eqref{speedBS}, we obtain the formula for the speed given in Theorem \ref{speedthm}. \end{proof}

\section{Proof of Lemma \ref{thelemma}}

Lemma \ref{thelemma} follows from Theorem 1.6 in \cite{kosygina}, but we offer the following direct proof, both in the interest of self-containment and because the techniques used in the proof will be referenced later. 

For $ \mathbf{p} = (p_1, p_2, \ldots, p_M )$, let $\mathbf{p}^{'} = (p_1, p_2, \ldots, p_M, p_{0}, p_{0}, ...,p_{0}) \in \mathbb{R}^{N}$, and choose $N > M$ such that $\delta(N,\mathbf{p}^{'}) > 2$. We now construct the backwards branching-like processes associated to two excited random walks, the first an excited asymmetric random walk $(X_{\infty,n})_{n \geq 0}$ with parameters $M, \mathbf{p}, p_0 $ for $ p_0  > 1/2 $, and the second walk $(X_{N,n})_{n \geq 0}$ a standard excited random walk with $N$ cookies and cookie vector $\mathbf{p}^{'} \in \mathbb{R}^{N}$. We will use a coupling argument with the backwards branching-like processes associated to $(X_{\infty,n})_{n \geq 0}$ and $(X_{N,n})_{n \geq 0}$ to prove the lemma. To construct these backwards branching-like processes, we first let $ \left( \xi_{i,j} \right)_{j \geq 1}$ be given by

\begin{equation}
\xi_{i,j} = \begin{cases}
\text{Bernoulli}(p_j) : 1 \leq j \leq M \\
\text{Bernoulli}(p_0) : j > M, \\
\end{cases}
\label{xidef}
\end{equation}
and let $ \left( \zeta_{i,j} \right)_{j \geq 1}$ be similarly defined sequences of random variables such that
\begin{equation}
\zeta_{i,j} = 
	\begin{cases}
    	\text{Bernoulli}(p_j) : 1 \leq j \leq M \\
    	\text{Bernoulli}(p_0) : M < j \leq N \\
        \text{Bernoulli}(\frac{1}{2}) : j > N.
    \end{cases}
    \label{zetadef}
\end{equation}

We define the random variables $ A_{i,k} $ and $ B_{i,k} $ to be the number of failures  before $ k $ successes in the sequences of Bernoulli random variables $ \left( \xi_{i,j} \right)_{j \geq 1}$ and $ \left( \zeta_{i,j} \right)_{j \geq 1}$, respectively. Then we let $(Z_{\infty,n})_{n \geq 0}$ and $(Z_{N,n})_{n \geq 0}$ be Markov chains with transition probabilities given by:
\begin{align*}
& \text{P}_{\infty,z} \left( Z_{\infty,0} = z \right) = 1,
\\
& \text{P}_{\infty,z} \left( Z_{\infty,n+1} = k | Z_{\infty,n} = j \right) = P \left( A_{n+1,j+1} = k \right),
\\
& \text{P}_{N,z} \left( Z_{N,0} = z \right) = 1, \text{ and}
\\
& \text{P}_{N,z} \left( Z_{N,n+1} = k | Z_{N,n} = j \right) = P \left( B_{n+1,j+1} = k \right).
\end{align*}
Then $(Z_{\infty,n})_{n \geq 0}$ and $(Z_{N,n})_{n \geq 0}$ are the backwards branching-like processes associated to $(X_{\infty,n})_{n \geq 0}$ and $(X_{N,n})_{n \geq 0}$, respectively. Since $ p_0 > 1/2 $, we can couple $ \left( \xi_{i,j} \right)_{j \geq 1}$ and $ \left( \zeta_{i,j} \right)_{j \geq 1}$ such that:

\begin{itemize}
\item For all $ i,j $, $ \xi_{i,j} $ are independent Bernoulli random variables as defined in \eqref{xidef}.
\item For all $ i,j $, $\zeta_{i,j}$ are independent Bernoulli random variables as defined in \eqref{zetadef}.
\item $ \text{P}( \xi_{i,j} \geq \zeta_{i,j} ) = 1 $ for all $ i,j $.
\end{itemize}
One such coupling would be
\begin{align*}
\text{P} ( \xi_{i,j} = 1, \zeta_{i,j} = 1 ) = &
	\begin{cases}
    	p_i \text{ if } 1 \leq i \leq M \\
        p_0 \text{ if } M < i \leq N \\
        1/2 \text{ if } i > N,
    \end{cases}
\\
\text{P} ( \xi_{i,j} = 1, \zeta_{i,j} = 0 ) = &
	\begin{cases}
    	0 \text{ if } 1 \leq i \leq N \\
        p_0 - 1/2 \text{ if } i > N,
    \end{cases}
\\
\text{P} ( \xi_{i,j} = 0, \zeta_{i,j} = 0 ) = &
	\begin{cases}
    	1 - p_i \text{ if } 1 \leq i \leq M \\
        1 - p_0 \text{ if } i > M.
    \end{cases}
\end{align*}

Since $ \text{P}( \xi_{i,j} \geq \zeta_{i,j} ) = 1 $, it is clear from the definitions of $A_{i,k}$ and $B_{i,k}$ that $A_{i,k} \leq B_{i,k}$ for all $i,k$. Also, by definition $B_{i,k}$ is nondecreasing in $k$. Therefore if $Z_{\infty,k} \leq Z_{N,k}$, then 
\begin{align}
 Z_{\infty,k+1} = A_{k+1,Z_{\infty,k}+1} \leq B_{k+1,Z_{\infty,k}+1} \leq B_{k+1,Z_{N,k}+1}= Z_{N,k+1}.
 \end{align}
 Since $Z_{\infty,0} = Z_{N,0}$ by construction, induction shows that $Z_{\infty,n} \leq Z_{N,n}$ for all $n$. 
 
 We now use the coupling to show that the stationary distribution of $(Z_{\infty,n})_{n \geq 0}$ exists. Let $T_{\infty,0}^{+}$ and $T_{N,0}^{+}$ be the times at which $Z_{\infty,n}$ and $Z_{N,n}$ first return to 0, respectively. That is, 
 \begin{align*}
& T_{\infty,0}^{+} = \inf \{n>0: Z_{\infty,n}=0\}, \\
& T_{N,0}^{+} = \inf \{n>0: Z_{N,n}=0\}.
 \end{align*}
 Because $Z_{\infty,n} \leq Z_{N,n}$ under our coupling, we have that $ Z_{N,n} = 0$ implies $Z_{\infty,n} = 0$, so $T_{0,\infty}^{+} \leq T_{0,N}^{+}$, and hence $\mathbb{E}_{0}[T_{0,\infty}^{+}] \leq \mathbb{E}_{0}[T_{0,N}^{+}]$. Therefore if $(Z_{N,n})_{n \geq 0}$ is positive recurrent, so is $(Z_{\infty,n})_{n \geq 0}$. We chose $N$ so that $(X_{N,n})_{n \geq 0}$ would have positive speed (equivalently, $\delta(N, \mathbf{p}^{'}) > 2$), and hence be transient, so we know that $(Z_{N,n})_{n \geq 0}$ has a stationary distribution $\pi_{N}$ (which is equivalent to being positive recurrent). Therefore $(Z_{\infty,n})_{n \geq 0}$ has a stationary distribution, which we will denote $\pi_{\infty}$.

A similar application of the coupling can be used to show
\begin{align*}
\mathbb{E}_{\pi_{\infty}} \left[ Z_{\infty,0} \right] \leq \mathbb{E}_{\pi_{N}} \left[ Z_{N,0} \right] < \infty,
\end{align*}

\noindent where the last inequality holds because $N$ was chosen such that $\delta(N,\mathbf{p}^{'}) > 2$, and hence $v(N,\mathbf{p}^{'}) > 0$.

\section{Monotonicity Properties of Excited Random Walks}

It is known that a standard excited random walk with parameters $M \geq 1$ and $\mathbf{p} \in \mathbb{R}^{M}$ is transient if and only if $\delta(M, \mathbf{p}) > 1$, and has positive speed if and only if $\delta(M, \mathbf{p}) > 2$. But while $\delta(M, \mathbf{p})$ completely determines these qualitative properties, it is known that the speed of a standard excited random walk is not a function of $\delta(M, \mathbf{p})$ \cite{holmes}. In Section 4.2 we give a new proof of this fact. Further, our argument can be used to show that $\delta(M, \mathbf{p})$ and $v(M, \mathbf{p})$ are unrelated when $\delta(M, \mathbf{p}) > 2$, in the sense that there exist excited random walks with arbitrarily large $\delta$ parameters and arbitrarily small speeds. Before proving this, we present some previous results on monotonicity.

\subsection{Previous Results on Monotonicity}

When considering vectors of cookie strengths $ \mathbf{p} = (p_1, p_2, p_3, \ldots, p_M ) $, a natural partial ordering between two cookie vectors $ \mathbf{p} $ and $ \mathbf{q} $ of length $M$ arises. If for all $ i = 1, \ldots, M $, $ p_i \leq q_i $, we write $ \mathbf{p} \leq \mathbf{q} $. Zerner  \cite{zerner} showed that if $ \mathbf{p} \leq \mathbf{q} $, then $ v(M, \mathbf{p}) \leq  v(M, \mathbf{q}) $. Holmes and Salisbury \cite{holmes} developed a weaker partial ordering for cookie vectors, generalizing the results from Zerner.
\begin{defn} \label{partialorder}

We write $ \mathbf{p} \preceq \mathbf{q} $  if there exists a coupling of $ ( \mathbf{Y}, \mathbf{Z}), \ \mathbf{Y} = (Y_1, Y_2, \ldots, Y_M), \ \mathbf{Z} = (Z_1, Z_2, \ldots, Z_M) $ such that
\begin{itemize}
\item $ \{ Y_1, Y_2, \ldots, Y_M \} $ are independent Bernoulli random variables with $ Y_i \sim \text{Bernoulli} \left(p_i \right) $.
\item $ \{ Z_1, Z_2, \ldots, Z_M \} $ are independent Bernoulli random variables with $ Z_i \sim \text{Bernoulli} \left(q_i \right) $.
\item $ \text{P} \left\{ \sum_{j=1}^m Y_j \leq \sum_{j=1}^m Z_j \right\} = 1 $ for all $ m = 1, 2, \ldots, M $.
\end{itemize}
Moreover, we write $ \mathbf{p} \prec \mathbf{q} $ if $ \mathbf{p} \preceq \mathbf{q} $ and $ \mathbf{p} \neq \mathbf{q} $.

\end{defn}

Under this partial ordering, if $ \mathbf{p} \preceq \mathbf{q} $, then $ v(M, \mathbf{p}) \leq  v(M, \mathbf{q}) $ and $\delta(M,\mathbf{p}) \leq \delta(M,\mathbf{q})$. If $ \mathbf{p} \prec \mathbf{q} $, then either $ v(M, \mathbf{p}) =  v(M, \mathbf{q}) = 0 $ or $ v(M, \mathbf{p}) <  v(M, \mathbf{q}) $, but importantly this strict partial ordering does not imply a strict inequality between $\delta(M,\mathbf{p})$ and $\delta(M,\mathbf{q})$ \cite{peterson, holmes}.

We take a moment now to discuss what these partial ordering techniques can and cannot show regarding the relationship between $\delta(M,\mathbf{p})$ and $v(M,\mathbf{p})$, and to describe the new monotonicity results given in Section 4.2. First, the strict partial ordering can be used to find $M, \mathbf{p}, \mathbf{q}$ such that $\delta(M,\mathbf{p}) = \delta(M,\mathbf{q})$, but $v(M,\mathbf{p}) < v(M,\mathbf{q})$, which shows that the speed is not a function of $\delta$ \cite{holmes}. Additionally, a continuity argument together with the above example gives $\delta(M,\mathbf{p +} \boldsymbol{ \epsilon}) > \delta(M,\mathbf{q})$, but $v(M,\mathbf{p}+\boldsymbol{\epsilon}) < v(M,\mathbf{q})$ for some $\boldsymbol{\epsilon} = (\epsilon, \epsilon, ..., \epsilon)\in \mathbb{R}^{M}$ \cite{peterson}. 
However, this argument cannot be used to produce a specific numerical example, since it is unknown how small $\epsilon$ must be. Furthermore, it is clear from the definition of the partial orderings that $\mathbf{p} \preceq \mathbf{q}$ implies $p_{1} \leq q_{1}$. Just as with the relationship between the speed and $\delta$, the strict partial ordering together with a continuity argument can show that there exist $M,\mathbf{p}, \mathbf{q}$ with $p_{1} > q_{1}$ and $v(M,\mathbf{q}) > v(M,\mathbf{p}) > 0$ \cite{peterson}, 
but again the proof is not constructive. Finally, the partial ordering techniques in general give information about the speed of excited random walks only in relation to each other and so cannot give any absolute information about the speed.


\subsection{Our Results on Monotonicity}

Throughout this section, we let $v(M,\mathbf{p})$ be the speed of a standard excited random walk with cookie vector $\mathbf{p} \in \mathbb{R}^{M}$, we let $v^{*}(p_{0},p_{1})$ be the speed of an excited asymmetric random walk with one cookie of strength $p_{1}$ and bias parameter $p_{0}$, and we let $v^{s}(p) = 2p-1$ be the speed of a simple random walk with parameter $p$ .We prove that $ v(M, \mathbf{p}) $ cannot be written as a function of $ \delta(M, \mathbf{p}) $ if $ \delta(M, \mathbf{p}) > 2 $ by proving a slightly more general theorem, which loosely speaking states that an excited random walk with a few strong cookies tends to move faster than an excited random walk with many weaker cookies.
\begin{thm}
	Choose $M \geq 3$ and $\mathbf{p}  = (p,p, \ldots , p) \in \mathbb{R}^{M}$ such that $\delta(M,\mathbf{p}) = M(2p-1) > 2$. For $i \in \mathbb{N}$ we define
	\begin{align}
	& p^{(i)} =  \frac{1}{2} + \frac{M(2p-1)}{2(M+i)} \label{p_idef},
	\\
	& \mathbf{p_{i}} \in \mathbb{R}^{M+i} = (p^{(i)},p^{(i)}, \ldots , p^{(i)}) \label{p_ivecdef},
	\end{align}
	
	so that $\delta(M+i, \mathbf{p_{i}}) = \delta(M,\mathbf{p})$ for all $i$. Then
	\begin{align}
	\lim_{i\rightarrow \infty} v(M+i, \mathbf{p_{i}}) = 0.
	\end{align}
	
	\label{deltathm}
\end{thm}
As $i$ increases, the number of cookies at each site increases and the strength of each cookie decreases in such a way that the ``total drift" at each site, as measured by the parameter $\delta$, is unchanged. That the speed should decrease as $i$ increases is intuitive, since as $i \rightarrow \infty$, the excited random walk acts more and more like a simple symmetric random walk, which has speed 0. The proof below is guided by this intuition.

\begin{proof}
A simple random walk with parameter $p^{(i)}$ is equivalent to an excited asymmetric random walk with one cookie of strength $p^{(i)}$ and bias parameter $p^{(i)}$. From the proof of Lemma \ref{thelemma}, it is clear that $v(M_i, \mathbf{p_i} ) \leq v^{*}(p^{(i)}, p^{(i)})$. Since $ v^{*}(p^{(i)}, p^{(i)}) = v^s (p^{(i)}) = 2p^{(i)} -1 $ where $v^s (p^{(i)})$ is the speed of a simple random walk with parameter $p^{(i)}$, we have $v(M_i, \mathbf{p_i} ) \leq 2p^{(i)} -1 $. As $i \rightarrow \infty$, $p^{(i)} \rightarrow \frac{1}{2}$, and hence $v_s(p^{(i)}) \rightarrow 0$. Therefore $v(M_{i},\mathbf{p_{i}})$ is a sequence of positive numbers bounded above by a sequence which tends to zero, and hence itself tends to zero.
\end{proof}

It is clear from Theorem \ref{deltathm} that two excited random walks with the same $\delta$ value need not have the same speed, and hence that the speed of an excited random walk cannot be expressed as a function of $\delta$.

We have a corollary which shows that it is possible to construct an excited random walk with parameters $M \geq 3, \mathbf{p} \in \mathbb{R}^M$ with $ \delta(M,\mathbf{p}) $ arbitrarily large and $v(M,\mathbf{p})$ arbitrarily small.

\begin{cor}
	Given any $\eta, \epsilon > 0 $, there exist $M \geq 3, \mathbf{p} \in \mathbb{R}^M$ such that $\delta(M,\mathbf{p}) > \eta$ and $v(M,\mathbf{p}) < \epsilon$.
	\label{smallspeed}
\end{cor}

Corollary \ref{smallspeed} easily follows from Theorem \ref{deltathm}, since $ \delta(M,\mathbf{p}) $ can be made arbitrarily large by increasing $ M $ and Theorem \ref{deltathm} shows how $v(M_{i},\mathbf{p_{i}}) $ can be made arbitrarily small.

We now give a corollary and example showing how to construct specific cookie vectors whose $\delta$ parameters and speed are in opposite relations. 

\begin{cor}
	\label{vdelta}
	There exist cookie vectors $\mathbf{q} \in \mathbb{R}^{3}, \mathbf{p_{i}} \in \mathbb{R}^{3+i}$ such that $v(3, \mathbf{q}) > v(3+i,\mathbf{p_{i}}) $ and $\delta (3, \mathbf{q}) < \delta (3+i,\mathbf{p_{i}}) $.
\end{cor}
\begin{proof}
	Let $M = 3$. Choose $q,p$ such that $\frac{5}{6} < q < p < 1$. Let $\mathbf{p} = (p, p, p) $,  $\mathbf{q} = (q, q, q) $, and define $p^{(i)}, \mathbf{p_i} $ as in Theorem \ref{deltathm}. While $v(3, \mathbf{q}$ cannot be calculated exactly for a cookie vector $\mathbf{q} = (q, q, q) $, there is a lower bound $f(q)$ on the speed $v(3, \mathbf{q})$ \cite{reu16} given by:
	\begin{align}
	v(3, \mathbf{q}) \geq f(q) = \frac{(6q-5)\left(q^2 - 2q - 1\right)}{24q^4 - 42q^3 - 3q^2 + 28q - 9}.
	\label{lowerbound}
	\end{align}
	Choose $N \in \mathbb{N}$ such that   
	\begin{align*}
	N \geq \frac{6 \left(p \left(24 q^4-42 q^3-3 q^2+28 q-9\right)+2 \left(-6 q^4+9 q^3+5 q^2-8 q+1\right)\right)}{(6 q-5) \left(q^2-2 q-1\right)}.
	\end{align*}
	Then for all $i > N$,
	\begin{align*}
	v^{s}(p^{(i)}) = 2p^{(i)} -1 = \frac{3(2p-1)}{(3+i)} < f(q) \leq v(3,\mathbf{q}),
	\end{align*}
	It is clear from the proof of Theorem \ref{deltathm} that $v(3+i, \mathbf{p_i} ) \leq v^{s}(p^{(i)}) $, so we $v(3+i, \mathbf{p_i} ) < v(3, \mathbf{q} ) $, and simple algebra also shows $ \delta ( 3, \mathbf{q} ) = 3 (2 q -1) < 3 (2 p -1) = \delta (3+i , \mathbf{p_i} ) $.
\end{proof}

\begin{exmp}
	Let $p = 0.99, q = 0.85$, $ N = 7 $. Let $\mathbf{p_{i}}$ and $p^{(i)}$ be defined as in Theorem \ref{deltathm}. Then $ p^{(8)} =  \frac{1}{2} + \frac{3(2*0.99-1)}{2*11} = \frac{697}{1100} $, and it is clear from the proof of Corollary \ref{vdelta} that $v(11, \mathbf{p_8} ) < v(3, \mathbf{q} ) $ while $ \delta (11 , \mathbf{p_8} ) = 2.94 > 2.1 = \delta ( 3, \mathbf{q} )  $.
\end{exmp}

Finally, we give a proposition and example that show how to construct two excited random walks with positive speeds whose initial cookies and speed are in opposite relations.
\begin{prop}
	\label{vconstruction}
	There exist cookie vectors $\mathbf{q} \in \mathbb{R}^{3},\mathbf{p} \in \mathbb{R}^{M}$ such that $q_{1} < p_{1}$, but $v(3,\mathbf{q}) > v(M,\mathbf{p}) > 0$.
\end{prop}

\begin{proof}
	Let $\frac{5}{6} < q < p < 1$ and $\mathbf{q} = (q,q,q) $. Then $ \delta(3, \mathbf{q} ) > 2 $. Choose $ \epsilon > 0 $ such that
	\begin{align*}
	\epsilon < \frac{(1-p)f(q)}{1-f(q)} = \frac{(1-p) (6 p-5) \left(q^2-2 q-1\right)}{2 \left(12 q^4-24 q^3+7 q^2+12 q-7\right)}
	\end{align*}
	where $f(q)$ is as defined in \eqref{lowerbound}. Then we have
	\begin{align*}
	v^{*} \left(1/2 + \epsilon,p \right) & = \frac{2(1/2 + \epsilon) - 1}{2(1/2 + \epsilon) - 1 + 2(1-p)}
	\\
	& = \frac{\epsilon}{\epsilon + (1-p)}
	\\
	& < f(q), 
	\end{align*} 
	so $ v^{*} \left(1/2 + \epsilon,p \right) < v(3, \mathbf{q} )$. Choose $N$ such that
	\begin{align*}
	N \geq \frac{4-2p}{1+2\epsilon}.
	\end{align*}
	
	Then for $ \mathbf{p} = (p, 1/2 + \epsilon, 1/2 + \epsilon, \ldots, 1/2 + \epsilon) \in \mathbb{R}^{M} $ with $ M > N $, we have $\delta(M, \mathbf{p} ) > 2 $. It is clear from the proof of Lemma \ref{thelemma} together with the values of $\delta(3,\mathbf{q})$ and $\delta(M,\mathbf{p})$ that $0 < v(M, \mathbf{p} ) < v^{*}(1/2 + \epsilon, p) $. Therefore, $ v(M, \mathbf{p} ) < v^{*}(1/2 + \epsilon, p) < v(3, \mathbf{q} ) $, whereas ${p}_{1} = p > q $ by construction.
\end{proof}

\begin{exmp}
	Let $p = 0.99, q = 0.85$. We choose $ \epsilon = 0.0045 $ and $ N = 114 $. It is clear from the proof of Proposition \ref{vconstruction} that if $M > N$, $\mathbf{p} = (p, 0.5+\epsilon, 0.5+\epsilon, \ldots, 0.5+\epsilon)  \in \mathbb{R}^{M}$, $\mathbf{q} = (q,q,q)$, we have $0 < v(M, \mathbf{p} ) < v(3, \mathbf{q} ) $ and ${q}_{1} < {p}_{1}$.
\end{exmp}

\subsection{Open Question}

Let $\mathbf{p_{i}}$ be described as in Theorem \ref{deltathm}. Theorem \ref{deltathm} shows that the terms $v(M+{i}, \mathbf{p_{i}})$ become arbitrarily small as $i$ increases, but do they do so monotonically? That is, does the inequality $  v(M+{i+1}, \mathbf{p_{i+1}}) \leq v(M+{i}, \mathbf{p_{i}}) $ hold for all $i$? It is in general not possible to use the partial ordering $ ( \preceq ) $ to answer this question, as is demonstrated by the following example. 

\begin{exmp}
	
	Let $M = 3$, $ p > 1/2 $, and $ \mathbf{p} = (p, p, p) $. Then $ p^{(1)} = \frac{6p+1}{8} \text{ and } p^{(2)} = \frac{3p+1}{5} $ by \eqref{p_idef}. In order to compare cookie vectors of the same length, we will consider 
	\begin{align*}
	& \mathbf{\tilde{p}_1} = (p^{(1)}, p^{(1)}, p^{(1)}, p^{(1)}, 1/2) \text{ and } 
	\\
	& \mathbf{p_2} = (p^{(2)}, p^{(2)}, p^{(2)}, p^{(2)}, p^{(2)} ).
	\end{align*}
	For the limit defined in Theorem \ref{deltathm} to be monotone, we must have $v(5,\mathbf{p_2}) \leq v(4,\mathbf{p_1}) $, which we can try to prove by showing $ \mathbf{p_2} \preceq \mathbf{\tilde{p}_1} $. But to have $ \mathbf{p_2} \preceq \mathbf{\tilde{p}_1} $, we must have that the probability of all $ p^{(2)} $ cookies being successes is less than the probability of four $ p^{(1)} $ and one $1/2$ cookie being successes, i.e. $ (p^{(2)})^5 \leq \frac{(p^{(1)})^4}{2} $. On the contrary, we have
	\begin{align*}
	(p^{(2)})^5 - \frac{(p^{(1)})^4}{2} = & \ \left(\frac{(3 p+1)}{5} \right)^5-\frac{1}{2} \left(\frac{(6 p+1)}{8} \right)^4
	\\
	& \ = \frac{9 (2p-1)^2 \left(55296 p^3+34956 p^2+7572 p+563\right)}{25600000}
	\\
	& \qquad > 0 \quad \text{for } p > 1/2.
	\end{align*}
	
	Since $ (p^{(2)})^5 > \frac{(p^{(1)})^4}{2} $, we have $ \mathbf{p_2} \npreceq \mathbf{\tilde{p}_1} $. Similarly, since $ p^{(2)} < p^{(1)} $, we have $ \mathbf{p_2} \nsucceq \mathbf{\tilde{p}_1} $, so we cannot determine any relationship between $v(5,\mathbf{p_2})$ and $v(4,\mathbf{p_1})$ by this partial ordering. 
	
\end{exmp}

\newpage
\bibliography{excitedAsymmetric}{}

\newpage
\appendix
\appendixpage
\section{Computing $\pi(0)$ and $\pi(1)$}

We attempted to compute $\pi(0)$ and $\pi(1)$, where $\pi$ is the stationary distribution of the backwards branching-like process associated to an excited asymmetric random walk with one cookie of strength $p_{1}$ and bias parameter $p_{0}$. For an excited asymmetric random walk $(X_{n})_{n \geq 0}$ which is transient to the right, we have $\pi(0) =  P(U_{0}^{\infty} = 0)$ and $\pi(1) = P(U_{0}^{\infty} = 1)$, where $U_{0}^{\infty}$ is the total number of steps from $0$ to $-1$ during the walk. This fact follows from the equality in distribution of the processes $(U_{x}^{n})_{0 \leq x \leq n}$ and $(Z_{n})_{n \geq 0}$, and the fact that the limiting distribution of a Markov Chain is equal to its stationary distribution. 

Throughout this analysis, it will be helpful to know the probability $P_{z}(T_{x} < T_{y})$ in a simple random walk with parameter $p$, where $T_x$ is as defined in \eqref{tndef}. The solution to this problem, often called the Gambler's Ruin problem, is known to be given by the function

\begin{align*}
h(p,x,y,z) = \begin{cases} 
\frac{1-(\frac{1-p}{p})^{z-y}}{1-(\frac{1-p}{p})^{x-y}}, & p \neq \frac{1}{2} \\
\frac{z-y}{x-y}, & p = \frac{1}{2}. 
\end{cases}
\end{align*}
We will use our interpretation of the stationary distribution in terms of the random variable $U_{0}^{\infty}$ together with the function $h$ to investigate $\pi(0)$ and $\pi(1)$. We have
\begin{align*}
\pi(0) = P(\text{walker never steps from 0 to -1}) = \prod_{k = 0}^{\infty} P\left(\inf_{T_{k}\leq n \leq T_{k+1}}X_{n} > -1 \right).
\end{align*}
We will condition each probability in the infinite product above on $X_{T_{k}+1}$ using the following probabilities:
\begin{align*}
&P(X_{T_{k} + 1} = k+1) = p_{1}, \\
&P(X_{T_{k} + 1} = k-1) = 1 - p_{1}, \\ 
&P\left(\inf_{T_{k}\leq n \leq T_{k+1}}X_{n} > -1 | X_{T_{k} + 1} = k+1\right) = 1, \text{ and} \\
&P\left(\inf_{T_{k}\leq n \leq T_{k+1}}X_{n} > -1 | X_{T_{k} + 1} = k-1\right) = h(p_{0}, k+1, -1, k-1).
\end{align*}
All of these equalities are clear except the last, which holds because after the walker steps down from $k$ at time $T_{k}$, there are no cookies at any site $j \in \{0,1,...,k\}$. Now, conditioning $P\left(\inf_{T_{k}\leq n \leq T_{k+1}}X_{n} > -1 \right)$ on $X_{T_{k}+1}$ yields
\begin{align}
\prod_{k = 0}^{\infty} P\left(\inf_{T_{k}\leq n \leq T_{k+1}}X_{n} > -1 \right) = \prod_{k=0}^{\infty}\{p_{1}+(1-p_{1}) h(p_{0}, k+1,-1,k-1)\},
\end{align}
To compute $\pi(1)$, we use the interpretation $\pi(1) = P(U_{0}^{\infty} = 1)$ to determine
\begin{align}
\pi(1) = \sum_{k = 0}^{\infty} P(\text{step left from 0 once between } T_{k} \text{ and }T_{k+1}\text{, nowhere else}).
\end{align}
We further note that the probability of stepping left nowhere except possibly between $T_{k}$ and $T_{k+1}$ is related to $\pi(0)$ by the equation
\begin{align}
\nonumber &P(\text{no left steps from except possibly between } T_{k} \text{ and }T_{k+1})
\\
&=\prod_{j \geq 0, j \neq k} P\left(\inf_{T_{j}\leq n \leq T_{j+1}}X_{n} > -1  \right) = \frac{\pi(0)}{p_{1} + (1-p_{1})h(p_{0,k+1,-1,k-1})}.
\end{align}
Now we observe that the probability of the walker stepping left from zero once between $T_{k}$ and $T_{k+1}$ and nowhere else is equal to the probability that he does not step left from $0$ at any time not between $T_{k}$ and $T_{k+1}$ multiplied by the probability that he steps left exactly once between $T_{k}$ and $T_{k+1}$ given he has not stepped left from zero elsewhere. Further, the probability that the walker steps left exactly once between $T_{k}$ and $T_{k+1}$ given he has not stepped left from zero elsewhere is given by the probability that he reaches $-1$ between $T_{k}$ and $T_{k+1}$ exactly once. Mathematically,
\begin{align*}
P(\text{walker steps left from 0 exactly once between } T_{k} \text{ and } T_{k+1}) 
\end{align*}
\begin{align}
= P(\# \{T_{k} \leq n \leq T_{k+1} : X_{n} = -1\} = 1),
\end{align}
which we can compute exactly using gambler's ruin probabilities:
\begin{align}
\nonumber &P(\# \{T_{k} \leq n \leq T_{k+1} : X_{n} = -1\} = 1)
\\
&= (1-P(\text{no left steps from 0 between } T_{k} \text{ and } T_{k+1}))h(p_{0},k+1,-1,0)
\end{align}
Now combining (A.2), (A.3), (A.4), and (A.5) yields the equation  
\begin{align}
\pi(1) = \sum_{k = 0}^{\infty} \pi(0) \frac{(1 - p_{1} - (1-p_{1})h(p_{0},-1,k+1,k-1))h(p_{0},k+1,-1,0)}{p_{1}+(1-p_{1})h(p_{0}, k+1,-1,k-1)}
\end{align}
Unfortunately, while (A.1) and (A.6) are explicit, they are too difficult to simplify even with the help of software - a fact which highlights the complexity of the backwards branching-like process and its stationary distribution.

\end{document}